\documentclass[11pt]{amsart}
\usepackage{amssymb,mathrsfs,graphicx,enumerate,hyperref}
\usepackage{amsmath,amsfonts,amssymb,amscd,amsthm,bbm}
\usepackage[retainorgcmds]{IEEEtrantools}
\usepackage{colortbl}
\usepackage{cite}

\title[Mean Field Game with velocity interactions]{A Cucker-Smale inspired\\ deterministic Mean Field Game\\ with velocity interactions}

\author[Santambrogio]{Filippo Santambrogio}
\address[Filippo Santambrogio]{\newline Univ. Lyon, Universit\'e Claude Bernard Lyon 1, \newline CNRS UMR 5208, Institut Camille Jordan, \newline 43
boulevard du 11 novembre 1918, F-69622 Villeurbanne, France\newline and Institut Universitaire de France}
\email{santambrogio@math.univ-lyon1.fr}
\author[Shim]{Woojoo  Shim}
\address[Woojoo Shim]{\newline The Research Institute of Basic Sciences, \newline Seoul National University, Seoul 08826, Republic of Korea(South Korea) }
\email{cosmo.shim@gmail.com}

\newtheorem{theorem}{Theorem}[section]
\newtheorem{lemma}{Lemma}[section]
\newtheorem{corollary}{Corollary}[section]
\newtheorem{proposition}{Proposition}[section]
\newtheorem{remark}{Remark}[section]

\newtheorem{definition}{Definition}[section]

\newcommand{\R}{\mathbb R}

\newcommand{\pical}{\mathcal{P}}

\newcommand{\ve}{\varepsilon}

\newcommand{\bbt} {\mathbb T}
\newcommand{\spt} {\mathrm{spt}}
\newcommand{\Lip} {\mathrm{Lip}}

\newcommand{\still} {\mathrm{Still}}

\begin{document}

\date{\today}



\begin{abstract}
We introduce a mean field game model for pedestrians moving in a given domain and choosing their trajectories so as to minimize a cost including a penalization on the difference between their own velocity and that of the other agents they meet. We prove existence of an equilibrium in a Lagrangian setting by using its variational structure, and then study its properties and regularity.
\end{abstract}

\maketitle \centerline{\date}

\section{Introduction}
\setcounter{equation}{0}
This paper aims at proposing a simple model which is a bridge between some collective motion models such as the well-studied Cucker-Smale model \cite{C-S}, mainly applied to flock behavior, and the theory of Mean Field Games (MFG for short), introduced in \cite{LL06cr1,LL06cr2,LL07mf} and independently in \cite{HCMieeeAC06}, for economical and engineering applications. 

The simplest Cucker-Smale model describes the evolution of a family $(x_i)_{i=1,\dots,N}$ of particles, which are typically meant to represent a bunch of birds. It is important to underline this bird interpretation since it is emblematic of the spirit of the model, in comparison with MFG: each particle is indeed socially influenced by the behavior of the others, but is not rational and does not really choose how to move but just ``follows the others''. The evolution followed by each $x_i$ is typically described by a Cauchy problem with initial datum on $x_i(0)$ and $x_i'(0)=v_i(0)$, and the equation has the form
$$x_i'={}{v_i},\quad v_i'=-\frac 1N\sum_j \eta(x_i-x_j)(v_i-v_j),$$
where $\eta(x)$ is an interaction kernel which is typically chosen as a decreasing function of $|x|$.
This models the fact that each particle tends to align their velocity to that of the other particles, with a weight depending on their distance. It has been widely studied after \cite{C-S}, see for instance \cite{H-Liu}.

On the other hand, MFG theory (see \cite{Carnotes} for the first set of lecture notes on the topic) is more concerned with the behavior of a family of rational agents. Hence, particles represent humans rather than animals, whose rationality is limited. A Mean Field Game consists in the following principle: a family of agents should choose (or control) their trajectory in a state space, optimizing a cost which involves the position of the other players. This gives rise to a non-cooperative game, for which we look for a Nash equilibrium (see \cite{N}). This game is what is usually called a differential game, which underlines its dynamic nature. The main difference between differential games and static games is that agents could deviate from a given strategy along time, and the other agents could react to this change.   In a static game where every agent should choose a trajectory, he/she writes down the desired trajectory on a piece of paper, puts it in a sealed envelop, and gives it to a notary, who will announce the cost to every player after opening all the envelops. In a differential game, the strategy which is written on the same piece of paper is of algorithmic nature: after saying how he/she will move first, a rule on how to move from his/her current location is given depending on the observation of what the others have done. This means that, when considering whether to deviate from a given strategy, each player should consider that if he/she changes his trajectory the other players will see it and adapt their own trajectories to his/her change. Yet, MFG focus on a case which simplifies a lot the study: the continuous case, where agents are supposed to be indistinguishable and negligible. Indeed, in this case, only the distribution of mass on the set of trajectories plays a role, and if a single agent decides to deviate, he/she does not affect this distribution, which means that the other agents will not react to his/her change. This boils down to a game where, indeed, each player just chooses a trajectory, writes it in a sealed envelop, and waits for knowing his/her output.
In this way, a configuration in MFG can be described by a measure on the set of possible curves (giving rise to the so-called {\it Lagrangian equilibria}, see for instance \cite{CanCap,M-S} but also \cite{CarJimSan} in a different framework), and each player tries to choose the best curve, minimizing a certain cost.

Typically, the optimization problem considered by each agent is of the form 
$$\min\left\{\int_0^T L(t,\gamma(t),\gamma'(t),Q)dt+\Psi(\gamma(T))\;:\; \gamma(0)=x_0\right\},$$
where $Q\in \pical(\Gamma)$ is the distribution of mass of the players on the space $\Gamma$ of possible paths (other possibilities, involving the final cost $\Psi$ also depending on $Q$, or the minimization of the time horizon $T$ needed to reach a given target, as in \cite{M-S}, are also considered).

In most of these models, the above cost takes into account the interactions between agents in such a way so as to penalize passing through regions with high concentration of players. In some widely-studied models (which have a variational structure, see for instance \cite{BenCarSan,LavSan,CIME} as well as many lectures in \cite{LionsCDF}),  we have $L(t,x,v,Q)=\frac 12|v|^2+g(\rho_t(x))$ where $\rho_t$ is the density of the distribution $(e_t)_\# Q$ of players at time $t$, and $g$ is an increasing function. In other models, called MFG of congestion, we have costs of the form $L(t,x,v,Q)=\rho_t(x)^\alpha|v|^\beta$, for some exponents $\alpha,\beta>0$. In some cases \cite{CarMesSan,modestproposal} the interaction takes the form of a constraint on the density, such as $\rho\leq 1$, which is meant to represent the fact that a crowd of agents cannot fit in a too small area, as a property inherited from a particular appraoch, where each agent is a rigid sphere, and spheres cannot overlap (this is the spirit of \cite{MauRouSan, survey-crowd}, whose goal is to give a continuous formulation to the ideas of \cite{crowd1,crowd2}). 

A natural question would be to consider a more refined version of this granular model: if the spheres are so dense that they touch each other, then they are stuck into a crystalline configuration where their velocity is also constrained. One could think at a model where $\rho>1$ is impossible, $\rho=1$ implies that locally the velocity should be constant, and $\rho\approx 1$ should impose that the velocity is almost constant. A description of this model could be done in terms of a kinetic representation, where the unknown $f(t,x,v)$ stands for the distribution of mass on the phase space and the above considerations would mean $f(t,x,\cdot)=\delta_{v(t,x)}$ when $\rho(t,x):=\int f(t,x,v)dv$ saturates the constraint. Of course, different models where the constraints are replaced by penalizations would also be possible. Yet, the kinetic description in terms of the local distribution of the velocity features a serious drawback: from the point of view of weak convergence, alternating the regions (in the spirit of homogenization) where the velocity is constant and equal to a certain vector $v$ and regions where it is equal to the opposite vector $-v$ is very close to having $f(t,x,\cdot) = \frac 12 \delta_{v(t,x)}+ \frac 12 \delta_{-v(t,x)}$, while their effects on the above scenarios are completely different: in one case each agent only meets agents with the same velocity, in the other at each location there is a struggle between agents with opposite velocity.
This lack of weak stability prevents any sort of local model from being well-posed. In particular, it will be impossible to prove any kind of existence result, either by variational methods or by fixed-point methods. This suggests to consider a model where the velocity of each agent interacts with that of other agents nearby, up to a certain positive distance, or with that of all agents, but weighting the interaction through a decreasing funtion of the distance. {}{This reminds a lot of the interaction among agents in the Cucker-Smale model, though it is more like a modeling issue rather than the well-posedness of the aforementioned Cauchy problem.} Forgetting about the constrained model and the explicit role played by the density that was presented above just as an example, a very simple MFG model could be built upon the assumption that the cost for the agent $i$ following the trajectory $x_i(t)$ should include a term of the form
\begin{equation}\label{main term}
\int_0^T \sum_j \frac12 \eta(x_i(t)-x_j(t))|x_i'(t)-x_j'(t)|^2 dt,
\end{equation}
where $\eta(z)$ is a decreasing function of $|z|$. It is important to observe that this cost includes, even if in a very mild and implicit form, a penalization for overcrowding. Indeed, a useful way of rewriting it is as follows: set 
\begin{equation*}
\begin{aligned}
a(t,x)&=\sum_j \eta(x-x_j(t)), \quad u(t,x)=\frac{\sum_j x_j'(t)\eta(x-x_j(t))}{a(t,x)},\\
\sigma(t,x)&=\sum_j \eta(x-x_j(t)) |x_j'(t)-u(t,x)|^2,
\end{aligned}
\end{equation*} 
and then observe that the above cost can be re-written as
$$\int_0^T \frac12\big[a(t,x_i(t))|x_i'(t)-u(t,x_i(t))|^2+\sigma(t,x_i(t))\big]dt.$$
The coefficient $a$ in front of the term penalizing the difference between the velocity of the agent $i$ and the local average velocity is indeed a coefficient taking into account how many agents are present nearby (it is in some sense a regularized version of the density $\rho(t,x)$). Of course one could imagine a more refined model where different terms depending on $a$ or on $\rho$ appear, but this is the simplest one that we can consider.

Let us make two observations about this model. The first is of modeling nature: the presence of non-local interactions, ruled by an interaction kernel $\eta$ which can be taken, for instance, of the form $\eta(z)=e^{-\frac{|z|}{\ve}}$ for some $\ve>0$ which determines the length scale of the interaction, {}{suggests} that this could be a reasonable model reproducing the well-known phenomenon of {\it lane formation}. Indeed, many studies, starting from experimental ones, about pedestrian motion show that agents tend in some situations to spontaneously form lanes walking in opposite directions (the reader can look, among many publications on this topic, at \cite{lane1,lane2,lane3}, for instance). This can be observed for instance in underground corridors, where agents going in opposite directions (those going to the train and those going out from a previous one) meet, and they spontaneously create a number of lanes where individuals go in a same direction. One of the explications for this phenomenon could of course be the fact that pedestrian try to avoid crossing individuals moving in opposite direction, since contact with them would slow down their motion more than contact with agents moving in the same direction, which is exactly what is penalized in our model. The width of each lane is influenced by many factors, and is usually of the order of few people size. It is a huge simplification to say this, but one could think that it should be related to the distance at which interactions are significant, i.e. to the charasteric size of the kernel $\eta$ (to $\ve$ in our example above). This explains why we believe that the model we present in this paper could be a good and simple choice to reproduce lane formation phenomena. This will be briefly {}{addressed} in the last section of the paper, devoted to the multipopulation case (of course, lane formation cannot occur if all players have the very same goal, and we need to observe a situation where two different groups meet), but the answer could only come through a deeper study or numerical simulations.

The second observation concerns a refinement of the model and some mathematical properties. Indeed, the cost above is that it is invariant under translation of all the trajectories by a common time-dependent vector, i.e. if we replace each $x_j(t)$ with $x_j(t)+v(t)$ the condition for each $x_i$ to be optimal given the other trajectories does not change. This also applies in presence of a final cost $\Psi(x_i(T))$, at least if $\Psi$ is an affine function. This invariance implies lack of compactness which could be fatal when proving the existence of an equilibrium. For this reason and also for modeling reasons it is a good choice to also add a cost on the kinetic energy of each player, which means that each agent tries to adapt its velocity to that of the others but at the same time to push so that the velocity is not too large as this requires an effort for everybody. The global cost to be considered can therefore be of the form 
$$\int_0^T \left(\frac\delta2 |x_i'(t)|^2+ \lambda\sum_j \frac12 \eta(x_i(t)-x_j(t))|x_i'(t)-x_j'(t)|^2\right) dt+\Psi(x_i(T)),$$
for some parameters $\delta,\lambda>0$. The same can be reformulated in a continuous setting using a probability measure $Q\in \pical(\Gamma)$ where $\Gamma$ is the set of all possible paths:
$$\int_0^T \left(\frac\delta2 |\gamma'(t)|^2+ \lambda\int_\Gamma\frac12 \eta(\gamma(t)-\omega(t))|\gamma'(t)-\omega'(t)|^2dQ(\omega)\right) dt+\Psi(\gamma(T)),$$
and also re-written in terms of $a,u$ and $\sigma$. 

Finding an equilibrium amounts then at solving a fixed point problem: using the same strategy as in \cite{CanCap} and \cite{M-S}, one defines a multi-valued operator $\mathcal O$ associating with every measure $Q\in\pical(\Gamma)$ the set of all probabilities on $\Gamma$ with $(e_0)_\# Q=m_0$ which are concentrated on optimal curves for $F(\cdot,Q)$, and loof for a fixed point $Q\in \mathcal O(Q)$ using Kakutani's Theorem (see \cite{Kak}). In order to prove existence of a fixed point,  compactness properties of this operator are needed, which are unfortunately difficult to prove. Indeed, even under the addition of the kinetic term $\delta\int|\gamma'|^2$, it has not been possible to find a quantity that could be used to prove compactness on $Q$ (for instance, the average kinetic energy $\int\int|\gamma'|^2dtdQ(\gamma)$) and which decreased when passing from a measure $Q$ to a measure $\tilde Q\in \mathcal O(Q)$.

Instead, existence of an equilibrium will be proven in this paper via a variational method, specific to this very setting. Indeed, it is possible to prove that minimizers of a suitable functional $\mathcal J(Q)$ are necessarily equilibria, and it will be exploited in the sequel.

This variational framework makes a big difference between our work and some other recent works on MFG with velocity interactions. This is indeed not at all the first paper connecting Cucker-Smale models to MFG, and the very first one is probably \cite{M-P-R}. In such a paper, agents solve a stochastic control problem: they control their own acceleration, which also includes a brownian part (i.e. they follow $dx=vdt,\,dv=udt+dB_t$, $x$ being their position, $v$ their velocity, and $u$ the control), and the cost they minimize include a term of the form \eqref{main term}. Note that standard MFG consider interactions between agents based on their position in the state space, which does not allow velocity interactions unless the velocity itself is considered as a state. This explains why a natural choice is to consider $(x,v)$ as the state, and allow agents to control their acceleration, a topic which has been recently investigated, for instance in \cite{AchManMarTch} and \cite{CanMen}. Moreover, we mention the second part of \cite{BarCar} which specifically {}{addresses} the Cucker-Smale setting as a limit of MFG problems with control on the accelleration.

Mean-Field Games where the state is the position, and the velocity (or the control) of all agents enters into the computation of the cost are usually called  {\it MFG of controls} (or {\it extended MFG}, see \cite{extended1,extended2}). This is a 
less studied and more involved class of MFG, where the interaction between an agent and the rest of the population depends on the distribution or the controls chosen by the population, or on the joint distribution of states and controls. One of the first paper on MFG of controls was \cite{CarLeh}, for financial applications, where the agents control their portfolio by deciding how much to sell or to buy of a stock. For a very good introduction to the topic of MFG of controls, a good choice will be the PhD thesis \cite{Ziad-thesis}, which will be available soon. The author of such a thesis considers indeed (see also \cite{AchKob} and \cite{Kob}) MFG of controls which are really inspired by the Cucker-Smale model for collective motion. The cost is not exactly the same which is considered in the present paper, and diffusion is considered, but the qualitative features of the model are the same, despite some technical difference. In particular, in order to develop a general existence theory for MFG of controls in a non-monotone setting (in many cases in MFG the monotonicity properties of the cost imply that agents prefer choosing locations which are not the same as the other players; when applying this to velocity interactions, this would mean that agents prefer to choose velocities which are not the same, differently from what Cucker-Smale based models consider), a certain contractivity property is required in \cite{Kob} and in Chapter 2 of  \cite{Ziad-thesis}, which means that the term $|x'(t)-u(t,x(t))|^2$ in the cost is replaced by $|x'(t)-\lambda u(t,x(t))|^2$, for $\lambda<1$.

Due to the analogies with some models presented in  \cite{Ziad-thesis}, the reader could conclude that  the model presented in our paper is also an example of MFG of controls. This is reasonable, but could be debated, even if classifying it or not as a MFG of controls is a matter of taste. Indeed, we believe that one of the key features of MFG of controls is that there is an interaction between players which is not only based on where they are, but also on their intentions (their control, the effort that they put in the motion, etc...). If their movement is completely and bijectively determined by their control, then the distinction between the position and the control become less clear. Indeed, for deterministic MFG (without Brownian diffusion), the distinction about what is the state of the system is quite subtle and debatable when players are negligible since, as we said, the game is finally statical and the choice of a strategy consists indeed in choosing a trajectory, which includes both the information about position and velocity. As a result, the name ``MFG of controls'' seems more adapted to the cases when the control of each player does not give a full information on the evolution of his/her position, as it happens in the presence of diffusion. 

As the reader can see, we insist on the difference between models with diffusion and deterministic models. Indeed, all the models considered so far in MFG with similar features as ours include diffusion, and it seems, to the best of our knowledge, this is the first paper with an existence result for a deterministic MFG of {}{this} type (we could also say, if we accept this terminology, for a deterministic MFG of control, even if Chapter 4 in \cite{Ziad-thesis} suggests that existence for monotone MFG of controls can also be proven with degenerate or no diffusion). Note that, in the framework of crowd motion and in connection with \cite{MauRouSan}, \cite{modestproposal} also introduced a model which later turned out to fit the framework of deterministic MFG of controls. More precisely, \cite{modestproposal} is an example of a MFG where the control of each agent does not give full information on his position, since the velocity of the agents is a function of their own control and of a global pressure effect which depends on all the controls and the positions. However, the same paper did not include a general existence result, which is out of reach because of mathematical difficulties. Indeed, the main reason why the model presented in this paper can be analyzed is the choice of a very precise algebraic form for the cost which allows for a variational forumaltion, in the sense that minimizers of a (non-convex) global energy can be proven to be equilibria, which allows for an existence result based on semicontinuity and not on fixed point theorems.

The present paper is hence organized as follows: after this long introduction, Section 2 gives a precise description of the model we consider and provides the definition of the equilibrium we seek, together with few preliminary results; Section 3 is devoted to the variational formulation: we introduce a certain global energy, prove that its minimizers are equilibria in the desired sense, and prove existence of minimizers; Section 4 is the most technical one, and analyzes the equilibria: in 4.1 we define the relevant quantities ($a$, $u$ and $\sigma$, already introduced in these pages) for the interaction, in 4.2 we formally derive the MFG system, coupling a Hamilton-Jacobi equation and a continuity equation, describing the equilibrium, in 4.3 we prove $C^{1,1}$ regularity of the trajectories followed by the agents, in 4.4 we apply this regularity to proving monokineticity (all agents passing through a same point a same time share the same velocity) and in 4.5 we prove that, for short time ($T$ small) only one curve is followed starting from each point, which can be interpreted as a pure-strategy equilibrium; finally, Section 5 briefly explains how to extend the model to the multipopulation case and its applications.
%

\section{The model: a Mean-Field Game with velocity interactions}

Even if in the introduction, in order to present the main ideas, we described the model as if a finite number of agents were involved, the rigorous description in this section will be done using the language of measures on curves, in order to include the continuous case where every agent can be negligible. The case of agents with positive mass can also fit this framework using measures with atoms, yet it is important to notice that 
\begin{itemize}
\item when agents have positive mass there is a difference between dynamic games and static games, and ignoring this difference means that we do not consider the dynamic aspect of the game;
\item as we consider arbitrary measures on the space of curves and we do not impose the possible mass of their atoms, an agent with positive mass can split its mass into several curves, which corresponds in Nash's language to mixed strategies.
\end{itemize}

In order to define our game and our notion of equilibrium, let us fix a space domain $\Omega$ which can be $\R^d$, a {compact} subset of $\R^d$, or the $d$-dimensional torus, and a time $T>0$. We call $\Gamma$ the space of all continuous curves defined on the interval $[0,T]$ and valued in $\Omega$, endowed with topology of uniform convergence. We denote by $e_t:\Gamma\to\Omega$ for $t\in[0,T]$ the evaluation map $e_t(\gamma)=\gamma(t)$. For curves $\gamma\in \Gamma$, we define their kinetic energy $K(\gamma)$ as
\[
K(\gamma):=\begin{cases}
\displaystyle\frac{1}{2}\int_{0}^T|\gamma'(t)|^2dt& \mbox{ if }\gamma\in H^1\\
\displaystyle \hspace{0.9cm}+\infty & \mbox{ if }\gamma\notin H^1
\end{cases}.\]

Then, given a function $\Psi:\Omega\to\R$ and a number $\delta>0$, we define a cost function on $\Gamma$ including the kinetic energy and $\Psi$ as a final cost:
\[
K_{\delta,\Psi}(\gamma):=\delta K(\gamma)+\Psi(\gamma(T)).\]
Given an interaction kernel $\eta:\R^d\to \R$, we also define a cost function on $\Gamma\times\Gamma$ via
\[ V(\gamma,\tilde\gamma):=\begin{cases}
\displaystyle \frac{1}{2}\int_{0}^{T}|\gamma'(t)-\tilde\gamma'(t)|^2\eta(\gamma(t)-\tilde\gamma(t)) dt & \mbox{ if }\gamma-\tilde\gamma\in H^1\\
\displaystyle\hspace{2.8cm}+\infty& \mbox{ if }\gamma-\tilde\gamma\notin H^1
\end{cases}. \]
Given a probability measure $Q\in\pical(\Gamma)$ and a parameter $\lambda>0$, we define $V_Q:\Gamma\to\mathbb{R}$ as 
\[V_Q(\cdot)=\int_{\Gamma} V(\cdot,\tilde\gamma) dQ(\tilde\gamma),\]   
and then
\begin{equation}\label{extendF}
F(\gamma,Q):=K_{\delta,\Psi}(\gamma)+{\lambda}V_Q(\gamma),
\end{equation}
a quantity which can take $+\infty$ as a value.

An important observation about the function $K_{\delta,\Psi}$ is the following.
\begin{lemma}\label{Kkin}
Suppose that $\Psi$ is either bounded or Lipschitz continuous. Then there exists a constant $C$, depending on $\Psi,\delta$ and $T$, such that we have
$$K_{\delta,\Psi}(\gamma)\geq \frac{\delta}{2}K(\gamma)+\Psi(\gamma(0))-C$$
for every curve $\gamma\in \Gamma$. 
\end{lemma}
\begin{proof}
The estimate is straightforward if $\Psi$ is bounded, since we of course have $\Psi(\gamma(T))\geq \Psi(\gamma(0))-C$. For $\Psi$ Lipschitz continuous, we use
$$\Psi(\gamma(T))\geq \Psi(\gamma(0))-\Lip(\Psi)\int_0^T|\gamma'(t)|dt\geq \Psi(\gamma(0))-{{}\left(\frac{\delta}{4}\int_0^T|\gamma'(t)|^2dt+\frac{\Lip(\Psi)^2}{\delta}T\right)}.\qedhere$$
\end{proof}

In the model we consider, the initial distribution of the agents is fixed, and identified by a probability measure $m_0\in\mathcal{P}(\Omega)$. We will hence consider the behavior of a population of agents, initially distributed according to $m_0$, and trying to minimize the function $\gamma\mapsto F(\gamma,Q)$ according to their trajectory distribution $Q\in\mathcal{P}(\Gamma)$. Their interaction gives rise to a game that we will call  $MFG(\Omega,\Psi,\delta,\eta, \lambda,m_0)$. The ingredient of this game are indeed
\begin{itemize}
\item A $d$-dimensional domain $\Omega$ which could be either $\R^d$, a compact subset of $\R^d$, or the $d$-dimensional torus, where the movement takes place;
\item A probability measure $m_0\in\pical(\Omega)$ on $\Omega$, standing for the initial distribution of players;
\item Two parameters $\delta,\lambda>0$ fixing the relative weight of the cost for the velocity of each player ($\delta$) and of the interaction cost ($\lambda$);
\item A cost function $\Psi:\Omega\to\R$ which stand for the final cost and is supposed to be continuous, and either bounded or Lipschitz continuous.
\item An interaction kernel $\eta:\R^d\to\R$, which weights the interaction between the velocities of the agents in terms of their distances, which is supposed to be continuous, positive, and bounded.
\end{itemize}
All the above assumptions on the data will not be repeated throughout the paper and just referred to as ``standing assumptions''. Only possible extra requirements for some precise parts of the paper will be precised. 

Once all the ingredients are set, we look for a Nash equilibrium, which is defined as follows.
%

\begin{definition}
	A measure $Q\in\mathcal{P}(\Gamma)$ is said to be an equilibrium of $MFG(\Omega,\Psi,\delta,\eta, \lambda,m_0)$ if $e_{0\#}Q=m_0\in\mathcal{P}(\Omega)$ and
		\begin{equation}\label{equil}
		\int_{\Gamma}F(\gamma,Q)dQ(\gamma)<\infty,\quad 
		F(\gamma,Q)=\inf_{\substack{w\in \Gamma\\\omega(0)=\gamma(0)}} F(\omega,Q),\quad \forall ~\gamma\in\spt(Q).
		\end{equation}
\end{definition}

In order to deal with the constraint on the initial distribution we define the set $\mathcal{Q}_{m_0}$, the subset of $\mathcal{P}(\Gamma)$ where the initial distribution is $m_0$:
\[\mathcal{Q}_{m_0}:=\left\{Q\in\mathcal{P}(\Gamma):e_{0\#}{Q}=m_0\right\}.\]

\noindent We here present several topological properties of the set $\mathcal{Q}_{m_0}$ for later use. The topology we use for probability measures, which we call weak convergence, consists in the convergence in duality with continuous and bounded functions.

\begin{proposition}\label{P2.1}
	For every $t\in[0,T]$, the push-foward mapping $\mathcal{P}(\Gamma)\ni Q \mapsto e_{t\#}Q\in\mathcal{P}(\Omega)$ is continuous for the weak convergence of probability measures.
\end{proposition}
\begin{proof} Let $(Q_n)_{n\geq 1}$ be a sequence of probability measures in $\mathcal{P}(\Gamma)$ that converges to $Q$. Then, for every bounded continuous function $f\in C_b(\Omega)$, we have 
	\[\begin{aligned}
	\lim_{n\to\infty}\int f d(e_{t\#}Q_n)=\lim_{n\to\infty}\int (f\circ e_t)dQ_n=\int (f\circ e_t )dQ=\int f d(e_{t\#}Q),
	\end{aligned} \]
	since $f\circ e_t\in C_b(\Gamma)$. Therefore, we have $e_{t\#}Q_n\overset{*}{\rightharpoonup} e_{t\#}Q$ and conclude the continuity of $e_{t\#}$.
\end{proof}
\begin{proposition}\label{P2.2}
	For every probability measure $m_0\in\mathcal{P}(\Omega)$, the set $\mathcal{Q}_{m_0}$ is non-empty, convex and closed in $\mathcal{P}(\Gamma)$. 
\end{proposition}

\begin{proof}
	The convexity of $\mathcal{Q}_{m_0}$ immediately follows from the linearity of $e_{0\#}:Q\to e_{0\#}Q$. To prove $\mathcal{Q}_{m_0}$ is non-empty, we consider the trivial mapping $\still:\Omega\to \Gamma$ where $\still(x)$ is the constant curve staying still at $x$ for each $x\in\Omega$, which is of course contained in $\Gamma$. Then, we can find an element $\still_{\#}m_0$ of $\mathcal{Q}_{m_0}$ since 
	\[\still_{\#}m_0\in\mathcal{P}(\Gamma),\quad e_{0\#}(\still_{\#}m_0)=m_0. \]
	Finally, one has the closedness of $\mathcal{Q}_{m_0}$ as a consequence of the continuity of $e_{0\#}:\mathcal{P}(\Gamma)\to\mathcal{P}(\Omega)$ in Proposition \ref{P2.1}.
\end{proof}
\section{Variational Framework} 

In this section we will see how equilibria for our MFG can be obtained by minimizing an energy among measures $Q$ on curves. This kind of variational problems are nowadays standard in calculus of variations, starting from Brenier's formulation for the incompressible Euler equation (see \cite{br}), before being popularized in optimal transport in particular for branched transport and congested traffic (see Chapter 4 in \cite{OTAM}). Applications of variational principle on measures to {}{equilibria} involving interaction costs are also not new, but usually performed in statical settings, as one can see from \cite{BlaMosSan}.



\begin{lemma}\label{L3.1}
Under the standing assumptions on the model, the extended real-valued functional  \[J:\Gamma\times 
	\Gamma\to\R\cup\{+\infty\},\quad  J(\gamma,\tilde\gamma)=K_{\delta,\Psi}(\gamma)+K_{\delta,\Psi}(\tilde\gamma)+\lambda V(\gamma,\tilde\gamma)\]
	is lower semicontinuous.
\end{lemma}
\begin{proof}
	Let $\{(\gamma_n,\tilde\gamma_n)\}_{n\geq 1}$ be any converging sequence in $\Gamma\times \Gamma$ to the limit $(\gamma,\tilde\gamma)$. Then, the lower semicontinuity of $J$ is provided once we show
	\begin{equation*}\label{claim}
	\liminf_{n\to\infty} J(\gamma_n,\tilde\gamma_n)\geq J(\gamma,\tilde\gamma).
	\end{equation*}
	Here, we assume that the left-hand side of \eqref{claim} is smaller than $\infty$, since otherwise the above inequality clearly holds. Then, by taking a subsequence if necessary, we further assume  $\{J(\gamma_{n},\tilde\gamma_{n})\}_{n\geq 1}$ is a converging sequence, i.e.,
	\begin{equation*}\label{miniseq}
	\lim_{k\to\infty}J(\gamma_{k},\tilde\gamma_{k})=\liminf_{n\to\infty} J(\gamma_n,\tilde\gamma_n).
	\end{equation*}
	Using the positivity of $V$, the fact that $\gamma_n(0)$ and $\tilde\gamma_n(0)$ converge, and Lemma \ref{Kkin}, one can see that the $H^1$-norm of $(\gamma_{n},\tilde\gamma_{n})_{n\geq 1}$ is bounded.
%
%
	As a consequence, the uniform limit $(\gamma,\tilde\gamma)$ is in $H^1([0,T];\Omega\times\Omega)$ and $(\gamma_n',\tilde\gamma_n')$ converges to $(\gamma',\tilde\gamma')$ weakly in $L^2([0,T];\Omega\times\Omega)$. Since the functional $J(\gamma,\tilde\gamma)$ is an integral functional involving continuous functions of $(\gamma(t),\tilde\gamma(t))$ and convex (quadratic) functions of $(\gamma'(t),\tilde\gamma'(t))$, standard semicontinuity results (see \cite{Giusti}, for instance) apply.

\end{proof}
\begin{remark}\label{r4.1}
	From Fatou's Lemma, we also obtain
	\[\gamma_n\to \gamma\Rightarrow\liminf_{n\to\infty}\int_{\Gamma}J(\gamma_n,\tilde\gamma)dQ(\tilde\gamma)\geq \int_{\Gamma} J(\gamma,\tilde\gamma)dQ(\tilde\gamma). \] 
\end{remark}
\noindent From now on, we define, for every $Q\in \mathcal{P}(\Gamma)$
\begin{align}\label{KVJ}
\begin{aligned}
\mathcal{K}_{\delta,\Psi}(Q)&:=\int_{\Gamma}K_{\delta,\Psi}(\gamma)dQ(\gamma),\\
\mathcal{V}(Q)&:=\int_{\Gamma\times \Gamma} V(\gamma,\tilde\gamma)d(Q\otimes Q)(\gamma,\tilde\gamma), \\
\mathcal{J}(Q)&:=\int_{\Gamma\times \Gamma}J(\gamma,\tilde\gamma)d(Q\otimes Q)(\gamma,\tilde\gamma)=2\mathcal{K}_{\delta,\Psi}(\tilde\gamma)(Q)+\lambda\mathcal{V}(Q).
\end{aligned}
\end{align}

Note that all the above functionals can take the value $+\infty$. Yet, we observe the following implication, valid under the assumption of Lemma \ref{Kkin}:
\begin{equation}\label{KimplV}
\mathcal{K}_{\delta,\Psi}(\tilde\gamma)(Q)<+\infty\Rightarrow \mathcal{V}(Q)<+\infty.
\end{equation}

Indeed, the inequality proven in Lemma \ref{Kkin} allows to deduce from $\mathcal{K}_{\delta,\Psi}(\tilde\gamma)(Q)<+\infty$ the finiteness of the average kinetic energy
$\int_\Gamma K(\gamma)dQ(\gamma)$,
and we can estimate $\mathcal V$ thanks to
\[\begin{aligned}
	2\mathcal{V}(Q)&=2\int_{\Gamma\times \Gamma} V(\gamma,\tilde\gamma)dQ(\gamma)dQ(\tilde\gamma) \\
	&=\int_{\Gamma\times \Gamma}\int_{0}^T|\gamma'-\tilde\gamma'|^2\eta(\gamma-\tilde\gamma)dtdQ(\gamma)dQ(\tilde\gamma)\\
	&\leq 2(\sup\eta)\int_{\Gamma\times \Gamma}\int_{0}^T\left(|\gamma'|^2+|\tilde\gamma'|^2\right)dtdQ(\gamma)dQ(\tilde\gamma)\\
	&=4(\sup\eta)\int_{\Gamma} K(\gamma)dQ(\gamma)<\infty,
	\end{aligned} \]

\begin{lemma}\label{mineq1}
	Suppose, besides the standing assumptions on the model, that $\eta$ is even function, i.e., $\eta(y)=\eta(-y)$ for all $y\in \mathbb{R}^d$, then any (local) minimizer $Q_0$ of $\mathcal{J}$ in $\mathcal{Q}_{m_0}$ satisfies 
	\begin{equation}\label{equil2}
	\int_{\Gamma} F(\gamma,Q_0)dQ(\gamma)\geq \int_{\Gamma} F(\gamma,Q_0)dQ_0(\gamma),\quad \forall~ Q\in\mathcal{Q}_{m_0}.
	\end{equation}
\end{lemma}
\begin{proof}
	First, from the definition of $Q_0$, the functional $\mathcal{J}$ has a finite evaluation at $Q_0$, i.e., $\mathcal{J}(Q_0)<\infty$. Suppose 
	that there exists a probability measure $Q\in\mathcal{Q}_{m_0}$ satisfying 
	\[\int_{\Gamma} F(\gamma,Q_0)dQ(\gamma)<\int_{\Gamma} F(\gamma,Q_0)dQ_0(\gamma)=\mathcal{K}_{\delta,\Psi}(\tilde\gamma)(Q_0)+\lambda\mathcal{V}(Q_0). \]
	Then, since $\mathcal{J}(Q_0)=2\mathcal{K}_{\delta,\Psi}(\tilde\gamma)(Q_0)+\lambda\mathcal{V}(Q_0)$ is finite, the above integrals are all finite.
	Now, as $\mathcal{Q}_{m_0}$ is convex, we set  $Q_\varepsilon:=(1-\varepsilon)Q_0+\varepsilon Q$ and employ the minimizing condition 
	\begin{equation}\label{minim}
	\mathcal{J}(Q_0)\leq \mathcal{J}(Q_\varepsilon),\quad \forall \varepsilon\in (0,1).
	\end{equation}
	We can further split the $\mathcal{J}$-value at $Q_\varepsilon$ by the order of $\varepsilon$ as below: 
	\begin{align}
	\begin{aligned}
	\frac{1}{2}\mathcal{J}(Q_\varepsilon)&=\mathcal{K}_{\delta,\Psi}(\tilde\gamma)(Q_\varepsilon)+\frac{\lambda}{2}\mathcal{V}(Q_\varepsilon)\\
	&=(1-\varepsilon)\mathcal{K}_{\delta,\Psi}(\tilde\gamma)(Q_0)+\varepsilon \mathcal{K}_{\delta,\Psi}(\tilde\gamma)(Q)+\frac{\lambda(1-\varepsilon)^2}{2} \mathcal{V}(Q_0)+\frac{\lambda\varepsilon^2}{2} \mathcal{V}(Q)\\
	&\hspace{0.3cm} +\lambda\varepsilon(1-\varepsilon)\int_{\Gamma\times \Gamma} V(\gamma,\tilde\gamma) d(Q_0\otimes Q)(\gamma,\tilde\gamma)\\
	&=:\mathcal{I}_0+\varepsilon \mathcal{I}_1+\varepsilon^2\mathcal{I}_2.
	\end{aligned}
	\end{align}
	Here, the zeroth order term $\mathcal{I}_0$ coincides with $\frac{\mathcal{J}(Q_0)}{2}$, and the first order term is 
	\begin{align}
	\begin{aligned}
	\mathcal{I}_1&=-\mathcal{K}_{\delta,\Psi}(\tilde\gamma)(Q_0)+\mathcal{K}_{\delta,\Psi}(\tilde\gamma)(Q)-\lambda\mathcal{V}(Q_0)+\lambda\int_{\Gamma\times \Gamma} V(\gamma,\tilde\gamma) d(Q_0\otimes Q)(\gamma,\tilde\gamma)\\
	&=-\int_{\Gamma}F(\gamma,Q_0) dQ_0(\gamma)+\int_{\Gamma}F(\gamma,Q_0) dQ(\gamma),
	\end{aligned}
	\end{align}
	which is finite and negative. Moreover, the second order term $\mathcal{I}_2= \frac{\lambda}{2}(\mathcal{V}(Q_0)+\mathcal{V}(Q))$ is finite (since $\mathcal{V}(Q_0)$ is finite and $\int_{\Gamma} F(\gamma,Q_0)dQ(\gamma)$ is finite, which implies $\mathcal K_{\delta,\Psi}(\tilde\gamma)(Q)<+\infty$ and hence $\mathcal{V}(Q)<+\infty$). This makes a contradiction for small $\ve>0$.
\end{proof}

\noindent Now, we show that a measure $Q_0$ satisfying $\mathcal{J}(Q_0)<\infty$ and \eqref{equil2} is indeed satisfying \eqref{equil}. First, we prove the following.

\begin{lemma}\label{mineq2}
	Under the standing assumptions on the model, if $Q_0\in\mathcal{Q}_{m_0}$ and 
	\begin{equation}\label{assump}
	\mathcal{J}(Q_0)<\infty,\quad 	\int_{\Gamma} F(\gamma,Q_0)dQ_{0}(\gamma)\leq \int_{\Gamma} F(\gamma,Q_0)dQ(\gamma),\quad \forall~ Q\in\mathcal{Q}_{m_0},
	\end{equation}
	then $Q_0$ satisfies 
	\begin{equation}\label{equil3}
	F(\gamma,Q)=\inf_{\substack{w\in \Gamma\\\omega(0)=\gamma(0)}} F(\omega,Q),\quad Q_0-\mbox{almost~every}~ \gamma.
	\end{equation}
\end{lemma}
\begin{proof}
	For any real number $p$, define $A^p$ as 
	\[A^p:=\left\{\gamma\in \Gamma: F(\gamma,Q_0)\leq p \right\}, \]
	which is clearly a $\sigma$-compact subset of $\Gamma$.	Now, for any positive rational number $q$, define a multifunction $\mathcal{S}^{q}$ as
	\[\mathcal{S}^{q}:\Gamma\to 2^{\Gamma},\quad \mathcal{S}^{q}(\gamma)=\begin{cases}
	\hspace{1.7cm}\emptyset& \gamma\notin H^1\\
	\left\{\omega\in A^q:\omega(0)=\gamma(0) \right\}& \gamma\in H^1.
	\end{cases} \]
	Then, for any closed ball $\overline{B}_R(x)$ in $\Gamma$, the lower inverse of $S^{q}$ is Borel measurable, since
	\begin{align}
	\begin{aligned}
	&\Big\{\gamma\in \Gamma:\mathcal{S}^{q}(\gamma)\cap \overline{B}_R(x)\neq \emptyset \Big\}\\
	&\hspace{1.7cm}=\Big\{\gamma\in H^1:\mathcal{S}^{q}(\gamma)\cap \overline{B}_R(x)\neq \emptyset \Big\}\\
	&\hspace{1.7cm}=\Pi_1\Big\{(\gamma,\omega)\in H^1\times A^q:\gamma(0)=\omega(0),~ \|\omega-x\|_{\sup}\leq R \Big\}\\
	&\hspace{1.7cm}=\bigcup_{p\in\mathbb{Q}\cap (0,\infty)}\Pi_1\Big\{(\gamma,\omega)\in A^p\times A^q:\gamma(0)=\omega(0),~ \|\omega-x\|_{\sup}\leq R \Big\},
	\end{aligned}
	\end{align}
	and $\Big\{(\gamma,\omega)\in A^p\times A^q:\gamma(0)=\omega(0),~ \|\omega-x\|_{\sup}\leq R \Big\}$ is compact in $\Gamma\times \Gamma$, as $H^1$ is compactly embedded into $\Gamma$. Here, we used the lower semicontinuity of $F(\cdot,Q_0)$ for the closedness of the above set. Therefore, the multifunction $\mathcal{S}^q$ is measurable. In particular, the set of points where its value is non-empty is also measurable.\\
	
	\noindent Now, assume that the set 
	\[\Big\{\gamma\in H:\exists~ \omega\in H,~ F(\omega,Q_0)<F(\gamma,Q_0),~\omega(0)=\gamma(0) \Big\} \]
	is not $Q_0$-negligible. This set is the union, among pairs of positive rational numbers $(q,r)$ satisfying $r>q$ of the sets
	\[\Big\{\gamma\in H: F(\gamma,Q_0)>r,~\left\{\omega: \omega(0)=\gamma(0),~F(\omega,Q_0)\leq q \right\}\neq \emptyset \Big\}, \]
	hence at least one of them has positive $Q_0$-measure.
	In other words, one can find a pair of positive rational number $(q,r)$ satisfying 
	\[Q_0(A_r^q)>0,\quad\mbox{where}\quad A_r^q:=\left\{\gamma\in \Gamma:F(\gamma,Q_0)>r \right\}\bigcap\left\{\gamma\in \Gamma:\mathcal{S}^q(\gamma)\neq\emptyset \right\}. \]
	
	\noindent Then, by using the Kuratowski-Ryll-Nardzewski Selection Theorem (see, for instance, \cite{C-V} for all the theory about measurable multifunctions), there exists a measurable selector $f:A_r^q\to \Gamma$ of the multifunction $\mathcal{S}^q:A_r^q\to 2^{\Gamma}$, i.e., $f$ is a measurable function satisfying 
	\[f(a)\in \mathcal{S}^q(a),\quad \forall~ a\in A_r^q. \]
	However, one can easily check that a measure $\widetilde{Q}:=f_{\#}(Q_0|_{A_r^q})+Q_0|_{\Gamma\setminus A_r^q}$ is in $\mathcal{Q}_{m_0}$ and 
	\[\int_{\Gamma} F(\gamma,Q_0)dQ_0(\gamma)>\int_{\Gamma} F(\gamma,Q_0)d\widetilde{Q}(\gamma), \]
	which contradicts \eqref{assump}.\end{proof}
The only difference between \eqref{equil} and \eqref{equil3} is that \eqref{equil} guarantees the minimality of $F(\gamma,Q)$ for all $\gamma$ in a support of $Q$, which is a more concrete and useful notion than only having it for $Q$-a.e. curve, as in \eqref{equil3} .
\begin{lemma}\label{mineq3}
	Condition \eqref{equil3} is equivalent to 
	\begin{equation*}
	\int_{\Gamma}F(\gamma,Q)dQ(\gamma)<\infty,\quad 
	F(\gamma,Q)=\inf_{\substack{w\in \Gamma\\\omega(0)=\gamma(0)}} F(\omega,Q),\quad \forall~\gamma\in\spt(Q).
	\end{equation*}
	\end{lemma}
\begin{proof}	For simplicity, we first give a proof in the case where $\Omega$ has no boundary (i.e. if it is the whole space $\R^d$ or $\bbt^d$. Let $(\gamma_n)_{n\geq 1}$ be a uniformly converging sequence of curves satisfying
	\[\gamma_n\to \gamma,\qquad F(\gamma_n,Q)=\inf_{\substack{w\in \Gamma\\\omega(0)=\gamma_n(0)}}F(\omega,Q)<\infty. \]
	Then, for any other curve $\tilde{\gamma}\in\Gamma$ with $\tilde{\gamma}(0)=\gamma(0)$, we have 
	\[F(\tilde{\gamma}_n,Q)\geq F({\gamma}_n,Q),\quad \forall n\geq 1, \] 
	where $\tilde{\gamma}_n$ is a translation of $\tilde\gamma$ and $\gamma_n(0)=\tilde{\gamma}_n(0)$. Moreover, since each $\tilde\gamma_n$ is a translation of $\tilde\gamma$, the time derivatives $\tilde\gamma_n'$ and $\tilde\gamma'$ are equal for all $n$. Therefore, all derivative terms in $F(\tilde{\gamma}_n,Q)$ are fixed and we can apply the dominated convergence theorem:
	\[\begin{aligned}
	\lim_{n\to\infty}F(\tilde\gamma_n,Q)=F(\tilde\gamma,Q).
	\end{aligned} \]
	For the upper bound of $F(\gamma,Q)$, we recall Remark \ref{r4.1} that the function $F(\cdot,Q)$ is lower semicontinuous in $\Gamma$: 
	\[\liminf_{n\to\infty} F(\gamma_{n},Q)\geq F(\gamma,Q). \]
	Hence, combining these two estimates, we can conclude that $\gamma$ is also a minimizer of $F(\cdot,Q)$ among all curves with same initial position. \\
	
	\noindent This construction cannot be applied as it is in the case of a bounded domain with boundary, as $\tilde\gamma_n$ could get out of the domain. In this case we propose a different construction, which could also be applied to the case without boundary but is slightly more involved. The starting point is the following: for every two points $x_0,x_1\in\Omega$ there exists a constant-speed geodesic $\sigma:[0,1]\to\Omega$ with $\sigma(0)=x_0,~\sigma(1)=x_1$ and $\int_0^1 |\sigma'(t)|^2dt= d_{\mathrm{geod}}(x_0,x_1)^2$. Moreover, the geodesic distance is equivalent, in smooth domains, to the Euclidean one, so that, if we fix a point $x_0$ and we replace $x_1$ by a sequence $x_n\to x_0$, we obtain a sequence of curves $\sigma_n$ with $\int_0^1 |\sigma_n'(t)|^2dt\to 0$. We apply this to the points $x_0=\gamma(0)$ and $x_n=\gamma_n(0)$ above and set $\delta_n:= \int_0^1 |\sigma_n'(t)|^2dt$. Then, we choose a sequence $\ve_n\to 0$, and we define
	$$\tilde\gamma_n(t):=\begin{cases} \sigma_n(1-t/\ve_n)&\mbox{ if }t\leq \ve_n\\
								\tilde\gamma(2(t-\ve_n))&\mbox{ if }\ve_n\leq t\leq 2\ve_n\\
								\tilde\gamma(t)&\mbox{ if }t\geq  2\ve_n\end{cases}.$$
We can see that $\tilde\gamma_n$ is a curve starting from $x_n=\gamma_n(0)$ and is hence a competitor for $\gamma_n$. We only need to prove that we have $\lim_n F(\tilde\gamma_n,Q)\leq F(\tilde\gamma,Q)$. Since $\tilde\gamma_n$ and $\tilde\gamma$ coincide on $[2\ve_n,T]$, this only amounts to show that the integral terms on $[0,\ve_n]$ and on $[\ve_n,2\ve_n]$ in the energy of $\tilde\gamma_n$ tend to $0$. Using
$$\int_0^{\ve_n}|\tilde\gamma_n'(t)|^2dt=\ve_n^{-1}\int_0^1 |\sigma_n'(t)|^2dt,\quad \int_{\ve_n}^{2\ve_n}|\tilde\gamma_n'(t)|^2dt=2\int_0^{2\ve_n}|\tilde\gamma'(t)|^2dt,$$
we see that we only need to choose $\ve_n\to 0$ such that $\delta_n/\ve_n\to 0$, which can be obtained by {}{choosing} $\ve_n=\sqrt{\delta_n}$.
\end{proof}

To sum up, we obtain the following:

\begin{proposition}\label{mineq4}
	Under the standing assumptions of the model, and supposing that $\eta$ is an even function, then any (local) minimizer $Q_0$ of $\mathcal{J}$ in $\mathcal{Q}_{m_0}$ is an equilibrium for the mean-field game $MFG(\Omega,\Psi,\delta,\eta,\lambda,m_0)$.   \end{proposition}
\begin{proof}
This is a consequence of Lemmas \ref{mineq1}, \ref{mineq2} and \ref{mineq3}.
\end{proof}
{}{We now prove the existence of such minimizer of $\mathcal{J}$. This can be done by considering a $\mathcal{J}$-minimizing sequence of measures $(Q_n)_{n\ge1}$ and showing its tightness. Then, the following lemma is necessary to complete this minimizing sequence argument. }
\begin{lemma}\label{L4.3}
	Let $X$ be any Polish space and $(Q_n)_{n\geq1}$ be a sequence of $\mathcal{P}(X)$ converges to $Q_\infty$. Then, $(Q_n\otimes Q_n)_{n\geq 1}$ weakly converges to $Q_\infty\otimes Q_\infty\in\mathcal{P}(X\times X)$ up to subsequence.
\end{lemma}
We provide a proof of this result, which is quite classical, as the easiest proof which are available in the literature (by density of suitably separable functions) are only adapted to the case of a compact space $X$.
\begin{proof}
	Since $X$ and $X\times X$ are Polish spaces, the sequential compactness is equivalent to the tightness and closedness, as a consequence of Prokhorov theorem. Thus, for every $\varepsilon>0$ there exists a compact set $K_\varepsilon\subset X$ such that 
	\[Q_n(X\setminus K_\varepsilon)<\varepsilon,\quad \forall ~n\geq 1. \]
	This allows to deduce tightness for the sequence $(Q_n\otimes Q_n)_n$, since 
	\[(Q_n\otimes Q_n)(X\times X\setminus K_\varepsilon\times K_\varepsilon)<2\varepsilon,\quad \forall Q_n\geq 1. \]
	 Therefore, by using Prokhorov theorem and taking a subsequence if necessary, a limit $\mu$ of $(Q_n\otimes Q_n)_{n\geq1}$ exists in $\mathcal{P}(X\times X)$, and we claim that $Q_\infty\otimes Q_\infty$ is equal to $\mu$, i.e., 
	\begin{equation}\label{claim2}
	\int_{X\times X}\phi ~d\mu=\int_{X\times X}\phi ~d(Q_{\infty}\otimes Q_{\infty}),\quad \forall~ \phi\in C_b(X\times X).
	\end{equation}	First, note that we may assume $\phi\geq 0$ in \eqref{claim2}, since $\phi$ is assumed to be bounded and both $\mu$ and $ Q_\infty\otimes Q_\infty$ are probability measures. Then for any $\varepsilon>0$, we consider a compact set $K=K_\varepsilon$ and an approximation $\overline\phi$ of $\phi|_{K\times K}$ satisfying 
	\[ \|\overline\phi-\phi\|_{L^\infty(K\times K)}<\varepsilon,\quad \overline\phi=\sum_{i=1}^{m} f_ig_i,\quad f_i,g_i~\mbox{continuous functions on}~K. \]
	This is possible thanks to a standard use of the Stone-Weierstrass theorem, but requires compactness.
Then, we have	
	\begin{align}\label{SW}
	\begin{aligned}
	\int_{X\times X} \phi~ d(Q_{n}\otimes Q_{n})&\leq 2\varepsilon\|\phi\|_{L^\infty}+\int_{K\times K} \phi~ d(Q_{n}\otimes Q_{n})\\
	&\leq 2\varepsilon\|\phi\|_{L^\infty}+\varepsilon+\int_{K\times K} \overline\phi~ d(Q_{n}\otimes Q_{n})\\
	&=2\varepsilon\|\phi\|_{L^\infty}+\varepsilon+\sum_{i=1}^m\left(\int_K f_i~dQ_{n} \right)\left(\int_K g_i~dQ_{n} \right),
	\end{aligned}
	\end{align}
	and as $(Q_{n}|_K)_{n\geq 1}$ is obviously a tight sequence of measures, it converges (up to subsequences) to a measure $\nu_K$, supported on $K$, and satisfying $\nu_K\leq Q_\infty$.
	Thus, we take the limit $n\to \infty$ in \eqref{SW} to deduce
	\[\begin{aligned}
		\int_{X\times X} \phi~ d\mu&\leq 2\varepsilon\|\phi\|_{L^\infty}+\varepsilon+\int_{K\times K}\overline\phi d (\nu_K\otimes \nu_K)\\
	&\leq 2\varepsilon\|\phi\|_{L^\infty}+2\varepsilon+\int_{K\times K}\phi d (\nu_K\otimes \nu_K)\\
	& \leq 2\varepsilon\|\phi\|_{L^\infty}+2\varepsilon+\int_{X\times X}\phi d (Q_\infty\otimes Q_\infty),
	\end{aligned} \]
	for every nonnegative bounded continuous function $\phi$ and $\varepsilon>0$. Sending $\ve\to 0$ and using the arbitrariness of $\phi$ we deduce 
	\[\mu\leq Q_\infty\otimes Q_\infty. \]
	Finally, since $\mu$ and $Q_\infty\otimes Q_\infty$ are probability measures in $X\times X$, we conclude $\mu=Q_\infty\otimes Q_\infty$. 	
\end{proof}

\begin{theorem}
Under the standing assumptions of the model, and supposing that $\eta$ is even function, then there exists a equilibrium probability measure on curves $Q_\infty$ in $\mathcal{Q}_{m_0}$ of the mean-field game $MFG(\Omega,\Psi,\delta,\eta,\lambda,m_0)$.  
\end{theorem}
\begin{proof}
	It suffices to show that the functional $\mathcal{J}$ restricted to $\mathcal{Q}_{m_0}$ has a minimizer. 
	
	
	Let us take a minimizing sequence $(Q_n)_{n\geq 1}\subset \mathcal{Q}_{m_0}$ of $\mathcal{J}$, i.e., 
	\[\mathcal{J}(Q_1)\geq \mathcal{J}(Q_2)\geq \cdots,\quad  \lim_{n\to\infty}\mathcal{J}(Q_n)=\inf_{Q\in \mathcal{P}(\Gamma)}\mathcal{J}(Q)<\infty.  \]
	The fact that the infimum is finite comes from  $\mathcal{J}(\still_{\#}m_0)=2\int_{\Omega}\Psi(x)dm_0(x)<\infty$.
	Now, we claim that the collection $\{Q_n:n\geq 1\}$ is tight. To verify this, we define  the set $\Gamma_M$ as
	\[\Gamma_M:= \left\{\gamma\in \Gamma: W(\gamma(0))\leq M,\quad K_{\delta,\Psi}(\gamma)\leq M \right\}, \] 
	where $W:\mathbb{R}^d\to(0,\infty)$ is a continuous function satisfying 
	\[\int_\Omega W(x)dm_0(x)=:C_W<\infty,\quad \left\{x:W(x)\leq b \right\} ~\mbox{is compact}~\forall~ b>0.  \] Then, $\Gamma_M$ is a bounded subset of $H^1$ for $H^1$-norm (since the bound on $W$ provides a bound on $\gamma(0)$ and the bound on $K_{\delta,\Psi}$ on the kinetic energy) and therefore $\Gamma_M$ is compact in $\Gamma$. 
	
	Now, we have 
	$$Q(\{\gamma\,:\,W(\gamma(0))>M\})=m_0(\{W>M\})\leq \frac{C_W}{M}$$
	and
	$$Q(\{\gamma\,:\,K_{\delta,\Psi}(\gamma)>M\})\leq \frac1M \int K_{\delta,\Psi}dQ\leq \frac{1}{2M}\mathcal J(Q).$$
	Therefore, 
	from Prokhorov theorem, the sequence $(Q_n)_{n\geq 1}$ has a converging subsequence $(Q_{n_i})_{i\geq 1}$ to some probability measure ${Q}_\infty\in\mathcal{P}(\Gamma)$, and this is indeed contained in $\mathcal{Q}_{m_0}$ by the closedness in Proposition \ref{P2.2}. 
	Moreover, by using Lemma \ref{L4.3}, the corresponding sequence of product measures $(Q_{n_i}\otimes Q_{n_i})$ converges to $Q_\infty\otimes Q_\infty$. Hence, we have 
	\[\mathcal{J}(Q_\infty)=\int_{\Gamma\times \Gamma} J d(Q_\infty\otimes Q_\infty)\leq \liminf_{n\to\infty} \int_{\Gamma\times \Gamma} J d(Q_n\otimes Q_n)=\inf_{Q\in \mathcal{P}(\Gamma)}\mathcal{J}(Q), \]
	and conclude that $Q_\infty$ is the desired minimizer, which is an equilirbium thanks to Proposition \ref{mineq4}.
\end{proof}

\section{Analysis of optimal curves and equilibria}

Throughout the section, we will occasionally need to add some assumptions to the standing ones, either on the kernel $\eta$ or on the domain $\Omega$. In part of the analysis we will indeed need to add the simplifying assumption that $\Omega$ has no boundary (i.e., either we are on the torus or on the whole space $\R^d$;\ of course other Riemannian manifold settings could be considered but we will stick to the Euclidean case) and/or that the communication weight $\eta$ is strictly positive and smooth, satisfying $|\nabla\eta|\leq C\eta$ for a finite constant $C$. This last assumption is of course satisfied in the compact case of the torus as soon as $\eta>0$ and $\eta$ is Lipschitz continuous. The final cost $\Psi$ can also be supposed to be Lipschitz continuous (instead of Lipschitz or bounded). On the other hand, it is not necessary for our analysis to suppose that $\eta$ is even, which is crucial for existence purposes but not for regularity. 

It is important to underline now that any equilibrium $Q$ that we analyze in this section is concentrated on $H^1$ curves, which will be useful in particular to obtain a.e. differentiability.
\subsection{Relevant macroscopic quantities}

For a better understanding of the problem we define the following quantities.

\begin{eqnarray*}
	M_1(t)&:=&\int_\Gamma |\omega'(t)|dQ(\omega),\\
	M_2(t)&:=&\int_\Gamma |\omega'(t)|^2dQ(\omega),\\
	a(t,x)&:=&\int_\Gamma \eta(x-\omega(t))dQ(\omega),\\
	u(t,x)&:=&\frac{1}{a(t,x)}\int_\Gamma \omega'(t)\eta(x-\omega(t))dQ(\omega),\\
	\sigma(t,x)&:=&\int_\Gamma |\omega'(t)-u(t,x)|^2 \eta(x-\omega(t))dQ(\omega).
\end{eqnarray*}

The quantities $M_1$ and $M_2$ represent the first and second moment of the velocity at time $t$, respectively, according to the distribution $Q$. It is important to note that we have $M_1^2\leq M_2$ and $\int_0^T M_2(t)dt<\infty$.

As far as $a,u$ and $\sigma$ are concerned, we note that $a(t,x)$ represent a local amount of mass around $x$, $u(t,x)$ the average velocity in the same neighborhood, and $\sigma(t,x)$ is a measure of how much the velocity is disordered around the same point. We will see the role that these quantities play in the minimization problem $\min_\gamma F(\gamma,Q)$. It is important to note that $u$ and $\sigma$ are only defined for those times $t$ such that $Q-$a.e. curve is differentiable at time $t$, which is true for a.e. $t$, since $Q$ is concentrated on curves which are almost a.e. differentiable. This is an easy consequence of the fact that the set of pairs $(t,\gamma)$ such that $\gamma$ is not differentiable at time $t$ is a Borel set in $[0,T]\times\Gamma$ and of a Fubini theorem.\smallskip

We start writing some simple inequalities about these functions, that we summarize in the following lemma.
\begin{lemma}
	We have the following inequalities on the values of $a,u,\sigma$, where $C$ denotes a universal constant, only depending on $\eta$.
	$$a\leq C,\quad a|u|\leq CM_1, \quad a|u|^2\leq CM_2,\quad \sigma\leq CM_2.$$
\end{lemma}
\begin{proof}
	The first inequality is trivial since $\eta$ is bounded, and the second comes is a similar way from the fact that we have $(au)(t,x)=\int_\Gamma \omega'(t)\eta(x-\omega(t))dQ(\omega)$. The inequality involving $|u|^2$ is less trivial, and we can proceed in the following way
	$$a|u|^2(t,x)=\frac{|(au)(t,x)|^2}{a(t,x)}=\frac{\left|\int_\Gamma \omega'(t)\eta(x-\omega(t))dQ(\omega)\right|^2}{\int_\Gamma \eta(x-\omega(t))dQ(\omega)}.$$
	Applying on the numerator the Cauchy-Schwartz inequality w.r.t. to the measure $\eta(x-\omega(t))dQ(\omega)$ to the functions $\omega\mapsto \omega'(t)$ and $\omega\mapsto 1$ we get
	$$\frac{\left|\int_\Gamma \omega'(t)\eta(x-\omega(t))dQ(\omega)\right|^2}{\int_\Gamma \eta(x-\omega(t))dQ(\omega)}\leq \frac{\int_\Gamma |\omega'(t)|^2 \eta(x-\omega(t))dQ(\omega)\cdot \int_\Gamma \eta(x-\omega(t))dQ(\omega)}{\int_\Gamma \eta(x-\omega(t))dQ(\omega)}\leq CM_2(t),$$
	where we finally use the boundedness of $\eta$.
	
	The inequality concerning $\sigma$ can be established using that
	$$\min_v \int_\Gamma |\omega'(t)-v|^2 \eta(x-\omega(t))dQ(\omega)$$
	is realized for $v=u(t,x)$, which is the average value of $\omega'$ for the measure  $\eta(x-\omega(t))dQ(\omega)$. Hence we have, taking $v=0$
	$$\sigma(t,x)=\int_\Gamma |\omega'(t)-u(t,x)|^2 \eta(x-\omega(t))dQ(\omega)\leq \int_\Gamma |\omega'(t)|^2 \eta(x-\omega(t))dQ(\omega)\leq CM_2,$$
	using once more the boundedness of $\eta$.
\end{proof}

Similar inequalities can also be established for the gradients of these functions, which share the same regularity as $\eta$ at least for those instants $t$ such that $M_1(t)$ and/or $M_2(t)$ are finite. For this result we will use the assumption on $\nabla \eta$.

\begin{lemma}
	Suppose that $\eta$ is strictly positive and Lipschitz continuous, and that there is a constant $C$ such that the inequality $|\nabla \eta(y)|\leq C\eta(y)$ holds for every $y\in \R^d$.
Then, we have the following inequalities on the values of the gradients of $a,u,\sigma$, where $C$ denotes a universal constant, only depending on $\eta$.
	$$\begin{aligned}
	&|\nabla_x a|\leq Ca\leq C,\quad |\nabla_x(au)|\leq Ca|u|\leq CM_1,\\
	&a|\nabla_x u|\leq Ca|u|\leq CM_1, \quad a|\nabla_x u|^2\leq CM_2,\quad |\nabla_x\sigma|\leq CM_2.
	\end{aligned}$$
\end{lemma}
\begin{proof}
	The first inequalities are trivial because of the bound $|\nabla \eta|\leq C\eta$. The second come in a similar way using  the fact that we have $\nabla_x(au)(t,x)=\int_\Gamma \omega'(t)\nabla\eta(x-\omega(t))dQ(\omega)$. In order to estimate $a|\nabla_x u|$ we write $a|\nabla_x u|\leq |\nabla_x(au)|+|u||\nabla_x a|$ and we use the previous estimates. This allows also to estimate $a|\nabla_x u|^2$ using
	$$a|\nabla_x u|^2=\frac{(a|\nabla_x u|)^2}{a}\leq C\frac{(a|u|)^2}{a}=Ca|u|^2\leq CM_2(t).$$
	
	The inequality concerning $\nabla_x\sigma$ can be established using 
	$$\begin{aligned}
	\nabla_x\sigma(t,x)&=\int_\Gamma |\omega'(t)-u(t,x)|^2 \nabla_x\eta(x-\omega(t))dQ(\omega)\\
	&\hspace{0.4cm}+2 \nabla_x u(t,x)\cdot\int_\Gamma(u(t,x)-\omega'(t)) \eta(x-\omega(t))dQ(\omega).
	\end{aligned}$$
	The first term in the right hand side can be estimated by boundedness of $\nabla_x\eta$ in terms of $CM_2(t)$, while for the second we use
	$$|\nabla_x u|\leq \sqrt{\frac{CM_2(t)}{a(t,x)}},\quad \int_\Gamma |u(t,x)-\omega'(t)| \eta(x-\omega(t))dQ(\omega)\leq \sqrt{\sigma(t,x) a(t,x)},$$
	together with the estimate $\sigma\leq CM_2(t)$ that we already proved.
\end{proof}

We finish this section by underlining the following computation, valid for any vector $v$:
\begin{multline*}\int_\Gamma |v-\omega'(t)|^2\eta(x-\omega(t))dQ(\omega)=\int_\Gamma |v-u(t,x)+u(t,x)-\omega'(t)|^2\eta(x-\omega(t))dQ(\omega)\\=\int_\Gamma |v-u(t,x)|^2\eta(x-\omega(t))dQ(\omega)+\int_\Gamma |\omega'(t)-u(t,x)|^2\eta(x-\omega(t))dQ(\omega),\end{multline*}
where the mixed term $\int_\Gamma (v-u(t,x))\cdot(u(t,x)-\omega'(t))\eta(x-\omega(t))dQ(\omega)$ vanishes since $u(t,x)$ is the average value of $\omega'$ for the measure  $\eta(x-\omega(t))dQ(\omega)$. We then get
$$\int_\Gamma |v-\omega'(t)|^2\eta(x-\omega(t))dQ(\omega)=a(t,x)|v-u(t,x)|^2+\sigma(t,x).$$

This allows to re-write the optimization problem for $\gamma$ using
\begin{equation}\label{Fausigma}
F(\gamma,Q)=K_{\delta,\Psi}(\gamma)+\frac\lambda 2\int_0^T \left(a(t,\gamma(t))|\gamma'(t)-u(t,\gamma(t))|^2+\sigma(t,\gamma(t))\right)dt.
\end{equation}

\subsection{A formal coupled system of PDEs} 
From the previous formulation in \eqref{Fausigma}, the optimization problem for $\gamma$ can be written as
$$\min \int_0^T L(t,\gamma(t),\gamma'(t);Q)dt+\Psi(\gamma(T)),$$
where the dependence on $Q$ only happens through $a$, $u$ and $\sigma$. More precisely, we have
$$L(t,x,v;Q)=\frac\delta 2|v|^2+\frac\lambda 2 \big(a(t,x)|v-u(t,x)|^2+\sigma(t,x)\big).$$
The Hamiltonian corresponding to $L$ is given by
$$H(t,x,p;Q):=\sup_v\; p\cdot v -L(t,x,v;Q)=\frac{|p+\lambda au|^2}{2(\delta+\lambda a)}-\frac\lambda 2(a|u|^2+\sigma),$$
where we omitted the dependence of $a,u$ and $\sigma$ on $(t,x)$. In the above maximization, the optimal $v$ for given $p$ is 
$$v= \frac{p+\lambda au}{\delta+\lambda a}.$$
Then, if we define the value function 
$$\varphi(t_0,x_0):=\min\left\{\int_{t_0}^T L(t,\gamma(t),\gamma'(t);Q)dt+\Psi(\gamma(T))\;:\;\gamma(t_0)=x_0\right\},$$
it is well known from classical dynamic programming arguments that $\varphi$ solves a Hamilton-Jacobi equation
\[\begin{aligned}
-\partial_t\varphi+\frac{\left|-\nabla \varphi+\lambda a u \right|^2}{2(\delta+\lambda a)}&=\frac{\lambda}{2}(a|u|^2+\sigma),\\
\varphi(x,T)&=\Psi(x),
\end{aligned} \]
and the optimal trajectory $\gamma$ solve
\[\gamma'(t)=v(\gamma(t),t),\quad v(x,t)=\frac{-\nabla \varphi(x,t)+\lambda (a u)(x,t)}{\delta+\lambda a(x,t)}. \]
On the other hand,  if we know the velocity field $v_t$ of agents for given agent density $\rho_t$, we can obtain the continuity equation:
\[\partial_t\rho+\nabla\cdot(\rho v)=0. \]
Therefore, once we assume that $Q$ is an equilibrium, the corresponding $\rho$ at time $t$ satisfies  
\[\rho_t(x)dx=d(e_{t\#}Q)(x), \]
and 
\[\begin{aligned}
a(x,t)&=\int_{\Gamma}\eta(x-\omega(t)) dQ(\omega)=(\rho\ast \eta)(x,t),\\
(au)(x,t)&=\int_{\Gamma}v (\omega(t),t)\eta(x-\omega(t))dQ(\omega)=\left((\rho v)\ast\eta\right)(x,t),\\
(a|u|^2+\sigma)(x,t)&=\int_{\Gamma}\left|v(\omega(t),t) \right|^2\eta(x-\omega(t))dQ(\omega)=\left((\rho|v|^2)\ast\eta\right)(x,t).
\end{aligned} \]
Combining all these equations, we can formally derive the following HJ-CE coupled PDE system for equilibrium of ${{}{MFG(\Omega,\Psi,\delta,\eta,\lambda,m_0)}}$:
\[\begin{cases}
\vspace{0.1cm}
\displaystyle-\partial_t\varphi+\frac{|-\nabla\varphi+\lambda a u|^2}{2(\delta+\lambda a)}=\frac{\lambda}{2}\left(a|u|^2+\sigma\right),\\
\displaystyle\partial_t\rho+\nabla\cdot(\rho v)=0,\quad 
v=\frac{-\nabla\varphi+\lambda au }{\delta+\lambda a},\\
a=\rho\ast \eta,\; au=(\rho v)\ast\eta,\; a|u|^2+\sigma=(\rho |v|^2)\ast\eta,\\
\vspace{0.2cm}
\displaystyle\varphi(x,T)=\Psi(x),\quad \rho_0=m_0.
\end{cases} \]

This system perfectly fits the framework described, for instance, in Section 1.1 of \cite{Ziad-thesis}, of course in the case with no diffusion ($\nu=0$). Indeed, we have an HJ equation on $\varphi$ which involves the three quantities $a,u$ and $\sigma$ depending on the joint {}{distribution} of positions and velocities of the players, and a continuity equation where the velocity field depends on $\nabla\varphi$ and on the same quantities. Then, we see that  $a,u$ and $\sigma$ are defined via an implicit equation involving them as well as $\rho$ and $\nabla\varphi$. 

\subsection{Some regularity results}

We start from writing the Euler-Lagrange equation for the minimization of $F(\gamma,Q)$. In all this sub-section we will make the following assumptions:
\begin{description}
\item[(H$\eta$)]
$\eta$ is strictly positive and Lipschitz continuous, and there is a constant $C$ such that the inequality $|\nabla \eta(y)|\leq C\eta(y)$ holds for every $y\in \R^d$;
\item[(H$\Omega$)] $\Omega$ has no boundary (i.e. it is either the torus or the whole space $\mathbb{R}^d$) ;
\item[(H$\Psi$)] $\Psi$ is Lipschitz continuous.
\end{description}
Using the expression \eqref{Fausigma} we have the following equation
\begin{equation}\label{EL}
\begin{aligned}
&\left(\delta \gamma'+\lambda a(t,\gamma) (\gamma'-u(t,\gamma))\right)'\\
&\hspace{1cm}={}{\frac{1}{2}}\Big[\nabla_x a(t,\gamma)|\gamma'-u(t,\gamma)|^2+2a(t,\gamma)(u(t,\gamma)-\gamma')\nabla_x u(t,\gamma)+\nabla_x\sigma(t,\gamma)\Big],
\end{aligned}
\end{equation}
coupled with the transversality condition 
$$\delta \gamma'(T)+\lambda a(T,\gamma(T)) (\gamma'(T)-u(T,\gamma(T)))=-\nabla\Psi(\gamma(T)).$$
Note that we can write these equations exploiting the fact that $\Omega$ has no boundary, otherwise some Lagrange multipliers would appear, making the estimates more complicated.

This has to be interpreted in a proper sense (in the spirit of the DuBois-Reymond Lemma): there exists an absolutely continuous function $z_\gamma$, such that
\begin{itemize}
\item $z_\gamma'$ equals the right hand side of \eqref{EL} a.e.: 
$$z_\gamma'={}{\frac{1}{2}}\Big[\nabla_x a(t,\gamma)|\gamma'-u(t,\gamma)|^2+2a(t,\gamma)(u(t,\gamma)-\gamma')\nabla_x u(t,\gamma)+\nabla_x\sigma(t,\gamma)\Big];$$
\item the final condition $z_\gamma(T)=-\nabla\Psi(\gamma(T))$ is satisfied;
\item $z_\gamma$ coincides a.e. with $\delta \gamma'(t)+\lambda a(t,\gamma(t)) (\gamma'(t)-u(t,\gamma(t)))$, which is indeed a function defined only for a.e. $t$.
\end{itemize}

Note that this requires that the right hand side of \eqref{EL} should be an integrable function, and its integrability is proven in Lemma \ref{unifzgamma} below.

For the sequel, we will fix a negligible set $N\subset [0,T]$ such that, for $t\notin N$ we have the following three properties
\begin{itemize}
	\item $Q$-a.e. curve is differentiable at $t$, which makes $u$ well-defined at such a time;
	\item $M_1(t)<+\infty$;
	\item at time $t$, the equality $z_\gamma(t)=\delta \gamma'(t)+\lambda a(t,\gamma(t)) (\gamma'(t)-u(t,\gamma(t)))$ is satisfied for $Q$-a.e. curve $\gamma$.
\end{itemize}

For $t\notin N$, we will also denote by $G(t)\subset H^1([0,T])\subset\Gamma$ the set of ``good'' curves $\gamma$ at time $t$:
$$G(t)=\{\gamma\in H^1([0,T])\,:\, \mbox{At time $t$, $\gamma$ is differentiable  and }z_\gamma=\delta \gamma'+\lambda a (\gamma'-u)\},$$
where we omitted the dependence of $a$ and $u$ on $(t,\gamma(t))$.
By definition of $N$, we have $Q(G(t))=1$ for all $t\notin N$.

\begin{lemma}\label{unifzgamma} 
Suppose, besides the standing assumptions on the model, that (H$\eta$), (H$\Omega$) and (H$\Psi$) hold. Then, if $Q$ is an equilibrium, for every curve $\gamma$ which is optimal for $F(\cdot,Q)$, the vector $z_\gamma$ is uniformly bounded by a common constant, only depending on $Q$ and on the parameters of the problem.
\end{lemma}
\begin{proof}
	The first point that we need to prove is a uniform bound on the energy $F(\gamma,Q)$ for optimal curves $\gamma$. To do this, we compare a curve $\gamma$ to the constant curve $\tilde\gamma=\still(x_0)$. We obtain then, ignoring a term with $\sigma(t,\gamma(t))\geq 0$,
	$$\begin{aligned}
	&\frac\delta 2\int_0^T|\gamma'(t)|^2dt+\frac\lambda 2\int_0^Ta(t,\gamma(t))|\gamma'(t)-u(t,\gamma(t))|^2dt\\
	&\hspace{2cm}\leq \frac\lambda 2\int_0^T\left(a(t,x_0)|u(t,x_0)|^2+\sigma(t,x_0)\right)dt+\left(\Psi(x_0)-\Psi(\gamma(T))\right).
	\end{aligned}$$
	We then use 
	$$\Psi(x_0)-\Psi(\gamma(T))\leq \Lip(\Psi)|\gamma(T)-\gamma(0)|\leq \Lip(\Psi)\int_0^T|\gamma'(t)|dt\leq \frac\delta 4\int_0^T|\gamma'(t)|^2dt+C(\delta,\Psi),$$
	which allows to write
	$$\frac\delta 4\int_0^T\!\!|\gamma'(t)|^2dt+\frac\lambda 2\int_0^T\!\!a(t,\gamma(t))|\gamma'(t)-u(t,\gamma(t))|^2dt\leq \frac\lambda 2\int_0^T\!\!\left(a(t,x_0)|u(t,x_0)|^2+\sigma(t,x_0)\right)dt+C.$$
	Finally, we use the bound $a|u|^2+\sigma\leq CM_2$ and the integrability of $M_2$ to obtain 
	\begin{equation}\label{comparison energy}
	\frac\delta 4\int_0^T|\gamma'(t)|^2dt+\frac\lambda 2\int_0^Ta(t,\gamma(t))|\gamma'(t)-u(t,\gamma(t))|^2dt\leq C.
	\end{equation}
	
	This implies that all curves which are optimal for $F(\cdot,Q)$ satisfy a uniform bound on both the kinetic energy and the term $\int_0^Ta(t,\gamma(t))|\gamma'(t)-u(t,\gamma(t))|^2dt$. 	
	We now proceed to estimating the integral in time of the right-hand side of \eqref{EL}, which would give boundedness of $z_\gamma$. The first term, using $|\nabla_x a|\leq Ca$ is easily seen to be integrable, and its integral is bounded by a universal constant, thanks to \eqref{comparison energy}. For the {}{second} term, we estimate 
	$$a|u-\gamma'||\nabla_x u|\leq \frac 12 a|u-\gamma'|^2+ \frac 12 a|\nabla_xu|^2,$$
	and both terms are integrable here because of \eqref{comparison energy} and of the inequality $a|\nabla_xu|^2\leq CM_2$. For the third term, we just need to use $|\nabla_x\sigma|\leq M_2$.
	
	This proves that $z_\gamma$ is bounded by a uniform constant.
\end{proof}

For the sake of Section \ref{convexity small T}, we also need to observe the following.
\begin{remark}
We obtained a bound for $z_\gamma$ when $\gamma$ is an optimal curve. We insisted that this bound is a universal constant, in the sense that it does not depend on the curve $\gamma$. We will denote this universal bound by $||z||_\infty$. This value could depend on $Q$ and on all the parameters of the problem ($\delta, \lambda, \Psi, \eta$, as well as $T$). Yet, it is easy to see, tracking all the possible dependencies on $T$ of the previous computations, that if the other parameters are fixed (in particular, $\delta,\lambda, \Psi,\eta$), then $||z||_\infty$ stays bounded as soon as $T$ is bounded (in particular, it does not degenerate if $T\to 0$).
\end{remark}

\begin{lemma} Suppose, besides the standing assumptions on the model, that (H$\eta$), (H$\Omega$) and (H$\Psi$) hold. Then, if $Q$ is an equilibrium, every curve $\gamma\in\spt(Q)$ is Lipschitz continuous with a Lipschitz constant at most $||z||_\infty/\delta$. \end{lemma}
\begin{proof}
	Consider an instant of time $t\notin N$. From the finiteness of $M_1(t)$ we deduce a bound on $(au)(t,x)$. This bound is uniform in $x$ but a priori not in $t$. From the boundedness of $z_\gamma$ (which is indeed also uniform in $t$) and of $au$ we deduce boundedness of $(\delta+{}{\lambda}a(t,\gamma(t)))\gamma'(t)$, i.e. of $\gamma'$, at least for those curves $\gamma\in G(t)$. We now set $L:=\sup |\gamma'(t)|<\infty$, the sup being taken among curves $\gamma\in\spt(Q)\cap G(t)$. We will prove $L\leq ||z||_\infty/\delta$.
	
	It is clear that we have $|u(t,x)|\leq L$ for every $x$, since the bound $|\gamma'(t)|\leq L$ is true for $Q-$a.e. curve $\gamma$. Then, from
	$$(\delta+{}{\lambda}a(t,\gamma(t)))|\gamma'(t)|\leq ||z||_\infty+{}{\lambda}a(t,\gamma(t))|u(t,\gamma(t))|\leq  ||z||_\infty+{}{\lambda}a(t,\gamma(t))L,$$
	we deduce 
	$$|\gamma'(t)|\leq \frac{ ||z||_\infty}{\delta+{}{\lambda}a(t,\gamma(t))}+\frac{{}{\lambda}a(t,\gamma(t))}{\delta+{}{\lambda}a(t,\gamma(t))}L.$$
	Suppose now $L>||z||_\infty/\delta$, in which case we can find $\varepsilon>0$ such that $||z||_\infty\leq \delta L (1-\varepsilon)$. Using the bound $a\leq C$, we then have
	$$|\gamma'(t)|\leq L \frac{ (1-\varepsilon)\delta +{}{\lambda}a(t,\gamma(t))}{\delta+{}{\lambda}a(t,\gamma(t))}=L\left(1-\frac{\delta\varepsilon}{\delta+{}{\lambda}a(t,\gamma(t))}\right)\leq L\left(1-\frac{\delta\varepsilon}{\delta+{}{\lambda}C}\right) .$$
	
	Taking the supremum over $\gamma$ we obtain a contradiction $L<L$.
	
	This shows $L\leq ||z||_\infty/\delta$ and hence every curve in $\spt(Q)$ satisfies a uniform bound on $\gamma'(t)$ for every $t$ such that $\gamma\in G(t)$ and $t\notin N$. This happens for almost every instant of time and it is enough, together with its absolute continuity, to state that $\gamma$ is Lipschitz continuous and its Lipschitz constant is at most $||z||_\infty/\delta$.
\end{proof}

\begin{lemma}\label{C11}
	Suppose, besides the standing assumptions on the model, that (H$\eta$), (H$\Omega$) and (H$\Psi$) hold. Then, if $Q$ is an equilibrium, every curve $\gamma\in\spt(Q)$ is $C^{1,1}$  with a uniform Lipschitz constant for $\gamma'$ only depending on $Q$ and on the parameters of the problem.\end{lemma}

\begin{proof}
We take two instants of time $t,s\notin N$ and we set $L:=\sup |\gamma'(t)-\gamma'(s)|<\infty$, the sup being taken among curves $\gamma\in\spt(Q)\cap G(t)\cap G(s)$. We want to prove a bound of the form $L\leq C|t-s|$. First we note that, from the previous Lipschitz uniform bound, we also deduce uniform bounds for the vector field $u$ and for the derivative of all functions $z_\gamma$. Hence we can write, for $\gamma\in\spt(Q)\cap G(t)\cap G(s)$,
\begin{multline*}
|(\delta+{}{\lambda}a(s,\gamma(s)))\gamma'(s)-{}{\lambda}(au)(s,\gamma(s)))-(\delta+{}{\lambda}a(t,\gamma(t)))\gamma'(t){}{+\lambda}(au)(t,\gamma(t)))|\\
= |z_\gamma(s)-z_\gamma(t)|\leq C|t-s|.
\end{multline*}
We now use the uniform Lipschitz bound of $a(t,x)$ and $(au)(t,x)$ w.r.t. $x$, together with the Lipschitz bound on $\gamma$ to deduce from the above inequality the following one 
$$|(\delta+{}{\lambda}a(s,\gamma(t)))\gamma'(s)-{}{\lambda}(au)(s,\gamma(t)))-(\delta+a(t,\gamma(t)))\gamma'(t)+{}{\lambda}(au)(t,\gamma(t)))|\leq C|t-s|.$$	
We then note that $a(t,x)$ is also Lipschitz in time, now that we know that all curves in $\spt(Q)$ are uniformly Lipschitz, since $|a(t,x)-a(s,x)|\leq \Lip(\eta)\int |\omega(s)-\omega(t)|dQ(\omega)$. This allows to obtain 
$$|(\delta+{}{\lambda}a(t,\gamma(t)))(\gamma'(s)-\gamma'(t))-{}{\lambda}(au)(s,\gamma(t)))+{}{\lambda}(au)(t,\gamma(t)))|\leq C|t-s|.$$	
We look now at the behavior in time of $(au)(t,x)$. We have
$$\begin{aligned}
|(au)(s,x)-(au)(t,x)|&\leq \Lip(\eta)\int |\omega(t)-\omega(s)|dQ(\omega)+\int \eta(x-\omega(t))|\omega'(s)-\omega'(t)|dQ(\omega)\\
&\leq C|t-s|+a(t,x)L,
\end{aligned}$$
where we used the fact that $Q$-a.e. curve $\omega$ satisfies $|\omega'(s)-\omega'(t)|\leq L$.
We then obtain
$$(\delta+{}{\lambda}a(t,\gamma(t)))|\gamma'(s)-\gamma'(t)|\leq C|t-s|+{}{\lambda}a(t,\gamma(t))L.$$
We now take a number $\ve>0$ and choose a curve $\gamma\in\spt(Q)\cap G(t)\cap G(s)$ such that $|\gamma'(s)-\gamma'(t)|\geq (1-\ve)L$, thus obtaining
$$(\delta-\ve(\delta+C))L\leq(\delta(1-\ve)-\ve {}{\lambda}a(t,\gamma(t)))L\leq C|t-s|,$$
which gives the desired bound as soon as one takes $\ve\to 0$.

With this bound in mind, we know that every curve $\gamma\in \spt(Q)$ is Lipschitz continuous, and its derivative, which is a priori only defined a.e., is Lipschitz continuous on a set of full measure. This is enough to conclude $\gamma\in C^{1,1}$.
\end{proof}

The above statement concerns the curves $\gamma\in \spt(Q)$. Hence, it does not necessarily apply to all optimal curves for $F(\cdot,Q)$. Indeed, it would be possible that some optimal curves do not belong to the support of $Q$, and the use of the set $G(t)$ and $N$ in the proof is suited for a proof targeting measures in $\spt(Q)$. Anyway, we can easily establish the following result.

\begin{corollary}\label{all optimal C11}
Suppose, besides the standing assumptions on the model, that (H$\eta$), (H$\Omega$) and (H$\Psi$) hold. Then, if $Q$ is an equilibrium, every curve $\gamma$ which is optimal for $F(\cdot,Q)$ for fixed initial point is $C^{1,1}$ and satisfies $|\gamma'|\leq ||z||_\infty/\delta$.
\end{corollary}
\begin{proof}
Once we know that all curves in $ \spt(Q)$ are $C^{1,1}$, we obtain the Lipschitz continuity (and boundedness, of course) in time and space of the functions $a(t,x)$ and $(au)(t,x), \sigma(t,x)$ (for $au$ we saw in the proof of Lemma \ref{C11} that the Lipschitz constant in time is the same as that of the velocities $\omega'$ for $\omega\in\spt(Q)$). We also remark that $|u|$ is bounded by a very explicit constant, i.e.  $|u|\leq ||z||_\infty/\delta$ since all curves in $\spt(Q)$ are Lipschitz continuous with this Lipschitz constant. Take now an optimal curve $\gamma$: from the fact that $z_\gamma$ is bounded (Lemma \ref{unifzgamma}), we deduce boundedness of $(\delta+{}{\lambda}a(t,\gamma(t))\gamma'(t)-{}{\lambda}(au)(t,\gamma(t))$. More precisely,  also using the bound on $|u|$, we obtain 
$$(\delta + {}{\lambda}a(t,\gamma(t)))|\gamma'(t)|\leq ||z||_\infty + {}{\lambda}a(t,\gamma(t))|u(t,\gamma(t))|\leq ||z||_\infty \left(1+\frac{{}{\lambda}a(t,\gamma(t))}{\delta}\right),$$
which implies $|\gamma'(t)|\leq ||z||_\infty /\delta$.

We then use again the properties of $z_\gamma$, together with the boundedness of $\gamma'$, to obtain that $z_\gamma$ is Lipschitz in time. Yet, since we know that $\gamma$ itself is Lipschitz in time, and that the functions $a$ and $au$ are Lipschitz in time and space, this provides $\gamma'\in \Lip$ and proves the claim.
\end{proof}

\subsection{Monokineticity}

It is useful to note that the $C^{1,1}$ result of Corollary \ref{all optimal C11} implies monokineticity in the following sense: if we take two curves $\gamma_1,\gamma_2\in \spt(Q)$, a time $t\in (0,T]$, and we suppose $\gamma_1(t)=\gamma_2(t)$, then we also have $\gamma_1'(t)=\gamma_2'(t)$. Hence, for each time $t$ which is not the initial time $t=0$, the velocity of all particles at a same point is the same, thus defining a velocity field $v(t,x)$ such that the curves $\gamma\in\spt(Q)$ follow $\gamma'(t)=v(t,\gamma(t))$ (without stating anything about the regularity of this velocity field $v$). For $t=T$ this is a consequence of the final condition in the Euler-Lagrnage equation $z_\gamma(T)=-\nabla\Psi(\gamma(T))$ and of the fact that $z_\gamma$ allows to identify $\gamma'$, once we know $a$ and $u$, which only depend on time and position. For $t<T$ this comes from regularity: should we have  $\gamma_1(t)=\gamma_2(t)$ but $\gamma_1'(t)\neq\gamma_2'(t)$, then we could build a curve $\tilde\gamma$ which is also optimal for $F(\cdot,Q)$, and coincides with $\gamma_1$ before time $t$, and with $\gamma_2$ after time $t$. This curve would not be $C^1$, and would hence violate Corollary \ref{all optimal C11}.

The above monokineticity allows to re-write our optimization problem using an Eulerian formulation in terms $\rho$ and $v$. Indeed, the problem of minimizing $\mathcal J $ becomes the minimization of
$$\frac\delta 2\int_0^T\int_\Omega|v_t|^2d\rho_t(x)dt+\frac\lambda 2\int_0^T\int_\Omega\int_\Omega \eta(x-x')|v_t(x)-v_t(x')|^2d\rho_t(x)d\rho_t(x')dt$$
among all $(\rho,v)$ satisfying 
\[\partial_t\rho+\nabla\cdot(\rho v)=0,~\rho_0=m_0. \]

As anyway our smoothness result is only valid when no boundary is present, we can ignore the boundary and re-write this in terms of convolutions, also setting, as it is usual in the Benamou-Brenier formulation of optimal transport, $w=\rho v$ : we get

$$\min\left\{\int_0^T\int_\Omega\left(\frac\delta 2\frac{|w_t|^2}{\rho_t}+\lambda\frac{|w_t|^2}{\rho_t} (\eta*\rho_t)-\lambda w_t\cdot (\eta*w_t)\right)dxdt\;:\;\partial_t\rho+\nabla\cdot w=0,\rho_0=m_0\right\}.$$

The reader can see that, thanks to the presence of the regularizing convolutional terms and of the initial term with the total kinetic energy, it would be possible to prove existence of a minimizer for the above problem by standard direct methods.

\subsection{Uniqueness of optimal curves and equilibria in pure strategies}\label{convexity small T}

The goal of this section is to show that, for small $T$, the functional $F(\gamma,Q)$ is strictly convex in $\gamma$ when restricted to the set of curves with a given Lipschitz constant and given initial point. We will see that this implies uniqueness of the minimizers and the existence of equilibria in pure strategies. To obtain this result, we need to assume lower bounds on $D^2\Psi$ and $D^2\eta$ (i.e., we suppose that $\Psi$ and $\eta$ are semi-convex, i.e. they become convex if we add to them a suitably large quadratic function)..

\begin{lemma}\label{hh' estimate}
Suppose that $\Psi$ and $\eta$ are $C^1$, Lipschitz continuous, and semi-convex, and that $Q$ is concentrated on curves which are all $C_0$-Lipschitz continuous. Then,  for every $L>0$, every curve $\gamma$ with $|\gamma'|\leq L$, and every $h\in H^1([0,T])$, we have
\begin{equation}\label{hh' eqn}
\begin{aligned}
F(\gamma+h,Q)&\geq F(\gamma,Q)+A[\gamma](h)+\frac\delta 2\int_0^T|h'(t)|^2dt\\
&\hspace{1cm}-C(L+C)\int_0^T|h(t)||h'(t)|dt-C(L+C)^2\int_0^T|h(t)|^2dt-C|h(T)|^2,
\end{aligned}
\end{equation}
where $A[\gamma](h)$ is a linear form in $h$ given by
\begin{eqnarray*}
A[\gamma](h)&=&\delta \int_0^T \gamma'(t)\cdot h'(t)dt+\lambda \int_0^T\int (\gamma'(t)-\omega'(t))\cdot h'(t) \eta(\gamma(t)-\omega(t))dQ(\omega)dt\\
&&+{{}\frac{\lambda}{2}} \int_0^T\int |\gamma'(t)-\omega'(t)|^2\nabla\eta(\gamma(t)-\omega(t))\cdot h(t)dQ(\omega)dt+\nabla\Psi(\gamma(T))\cdot h(T),
\end{eqnarray*}
and the constant $C$ depends on $C_0$ and on $\eta$.
\end{lemma}
\begin{proof}
We start from the following equalities or inequalities
\begin{align*}
\frac 12 \int_0^T|\gamma'+h'|^2dt&=\frac 12 \int_0^T|\gamma'|^2dt+\int_0^T\gamma'\cdot h'dt+\frac 12 \int_0^T|h'|^2dt,\\
\Psi(\gamma(T)+h(T))&\geq \Psi(\gamma(T))+\nabla\Psi(\gamma(T))\cdot h(T) -C|h(T)|^2,
\end{align*}
The part of $F$ which requires more attention is the following
\begin{align*}
&\int_0^T\int |(\gamma'+h')-\omega'|^2\eta(\gamma+h-\omega)dQ(\omega)dt\\
&\hspace{1cm}\geq \int_0^T\int |\gamma'-\omega'|^2\eta(\gamma+h-\omega)dQ(\omega)dt+2\int_0^T\int (\gamma'-\omega')\cdot h'\eta(\gamma+h-\omega)dQ(\omega)dt\\
&\hspace{1cm}\geq \int_0^T\int |\gamma'-\omega'|^2\eta(\gamma-\omega)dQ(\omega)dt+\int_0^T\int |\gamma'-\omega'|^2\nabla\eta(\gamma-\omega)\cdot h dQ(\omega)dt\\
&\hspace{1,5cm}- C\int |\gamma'-\omega'|^2dQ(\omega)|h|^2dt+2\int_0^T\int (\gamma'-\omega')\cdot h'\eta(\gamma-\omega)dQ(\omega)dt\\
&\hspace{1.5cm}-2\int_0^T\int |(\gamma'-\omega')|\cdot |h'|\Lip(\eta)|h|dQ(\omega)dt.
\end{align*}
Putting together these inequalities and using $|\gamma'|\leq L$ and $|\omega'|\leq C_0$ we obtain the desired result.
\end{proof}

\begin{proposition}
Suppose, besides the standing assumptions on the model, that (H$\eta$), (H$\Omega$) and (H$\Psi$) hold and that $\Psi$ and $\eta$ are $C^1$, Lipschitz continuous, and semi-convex. Then, if $T$ is smaller than a constant depending only on $\delta,\lambda,||z||_\infty, \eta$ and on the lower bounds of $D^2\Psi$ and $D^2\eta$ and if $Q$ is an equilibrium, for every initial point $x_0$ the problem
$$\min\{F(\gamma,Q)\,:\,\gamma(0)=x_0\}$$
has a unique solution.
\end{proposition}
\begin{proof}
We know from Corollary \ref{all optimal C11} that any optimal curve is necessarily $L-$Lipschitz continuous with $L=||z||_\infty$, so the result is proven if we prove, for instance, that $\gamma\mapsto F(\gamma,Q)$ is strictly convex on the set of $L-$Lipschitz continuous curves. For this, we use Lemma \ref{hh' estimate} since $Q$ is also concentrated on $L$-Lipschitz curves. In order to prove convexity it is enough to give conditions so that the second-order term appearing in \eqref{hh' eqn} is strictly positive for any function $h$ with $h(0)=0$ which is not identically $0$.
We use the following standard inequality which is valid for $h\in H^1$ with $h(0)=0$:
$$|h(t)|^2\leq \left(\int_0^T|h'(t)|dt\right)^2\leq T\int_0^T |h'(t)|^2dt.$$
We will use it for $t=T$ but also integrate, thus obtaining
$$\int_0^T|h(t)|^2dt\leq T^2 \int_0^T|h'(t)|^2dt$$
and
$$ \int_0^T|h(t)||h'(t)|dt\leq \left(\int_0^T|h(t)|^2dt\right)^{1/2} \left(\int_0^T|h'(t)|^2dt\right)^{1/2}\leq T \int_0^T|h'(t)|^2dt.$$
These inequalities allow to show the strict convexity of $F(\cdot,Q)$ if $T$ and $T(L+C)$ are small compared to $\delta$.\end{proof}

\section{Mean Field Game for multipopulations}

All the analysis in the previous sections has been developed in the simplified case where only one population of indistuingishable agents was moving. The only feature distinguishing the agents was their initial point. Yet, in applications the most interesting case is the one with several populations. Each population will be represented by a measure $Q_i\in \mathcal M_+(\Gamma)$ where $\mathcal M_+$ stands for the space of finite positive measures (indeed, there is no reason for the different populations to have the same mass, and it is hence not possible to normalize them to probability measures). Each population $i=1,\cdots,N$ will have its own target function $\Psi_i$ and its own mobility coefficient $\delta_i>0$ (if a population has smaller $\delta_i$, this means that its agents are less subject to effort costs for moving at high speed). On the other hand, the interaction cost (i.e. the cost due to the difference of the velocity of an agents w.r.t. that of the other agents) will be supposed to be the same for everybody, and will involve the interaction between each agent and every other agent, from the same population and from the others. In particular, we set $Q=\sum_i Q_i$ and all the agents of the population $i$ will try to minimize
$$F_i(\gamma,Q):=K_{\delta_i,\Psi_i}(\gamma)+V_Q(\gamma).$$
Given for each population an initial measure $m_{0,i}$, a target function $\Psi_i$, and a mobility parameter $\delta_i>0$, and given a common interaction weighting function $\eta$ and a parameter $\lambda>0$, we call the corresponding game $\mathrm{multi}MFG(\Omega,(\delta_i)_i, (\Psi_i)_i,\lambda,\eta, (m_{0,i})_i)$ and define a multipopulation equilibrium as follows:

\begin{definition}
	A family of measures $Q_i\in\mathcal{M}_+(\Gamma)$ is said to be an equilibrium of the game $\mathrm{multi}MFG(\Omega,(\delta_i)_i, (\Psi_i)_i,\lambda,(m_{0,i})_i)$ if $e_{0\#}Q_i=m_{0,i}\in\mathcal{M}_+(\Omega)$ and
		\begin{equation}\label{equil4}
		\int_{\Gamma}F_i(\gamma,Q)dQ(\gamma)<\infty\quad \forall i,\quad 
		F_i(\gamma,Q)=\inf_{\substack{w\in \Gamma\\\omega(0)=\gamma(0)}} F_i(\omega,Q),\quad \forall ~\gamma\in\spt(Q_i).
		\end{equation}
\end{definition}

The same analysis, with easy variants, as performed in the previous sections allows to prove the following facts
\begin{itemize}
\item If we define 
$$\mathcal J(Q_1,\dots,Q_N):=2\sum_i \int K_{\delta_i,\Psi_i}(\gamma)dQ_i(\gamma)+\lambda\mathcal V(\sum_i Q_i)$$
and $\eta$ is even, then every local minimizer of $\mathcal J$ in the set 
$$\{(Q_1,\dots,Q_N)\in \mathcal M_+(\Gamma)^N\;:\; (e_0)_\#Q_i=m_{0,i}\}$$ is an equilibrium;
\item for every equilibrium, $Q=\sum_i Q_i$ is concentrated on $C^{1,1}$ curves;
\item for each $i$ there exists a velocity field $v_i:[0,T]\times\Omega\to\R^d$ such that all curves $\gamma\in \spt(Q_i)$ satisfy $\gamma'(t)=v_i(t,\gamma(t)$ and we have monokineticity for each separate population.
\end{itemize}

It is in the framework of multipopulation equilibria that one can consider the question about lane formation that we mentioned in the introduction. For instance, one can take $\Omega=[-L,L]\times[0,1]\subset\R^2$ a long corridor, $N=2$ and $\Psi_1(x)=x_1,\Psi_2(x)=-x_1$, so that the agents of the two populations spontaneously move in opposite directions of the corridor. If $\eta(z)=e^{-|z|/\ve}$, one would expect the formation of lanes with (approximate) segregation of the two populations, the width of these lanes being of order $\ve$.

\bibliographystyle{amsplain}

\end{document}